\documentclass[11pt,reqno]{amsart}

\newcommand\version{February 11, 2019}

\usepackage{amsmath,amsfonts,amsthm,amssymb,amsxtra,graphicx}
\usepackage{tikz}

\usepackage{color}

\newtheorem{theorem}{Theorem}[section]

\newtheorem{corollary}[theorem]{Corollary}

\theoremstyle{definition}

\theoremstyle{remark}

\newtheorem{remark}[theorem]{Remark}

\numberwithin{equation}{section}

\newcommand{\C}{\mathbb{C}}

\renewcommand{\epsilon}{\varepsilon}

\newcommand{\N}{\mathbb{N}}

\begin{document}

\title[Sharpened Triangle Inequality --- \version]{Inequalities  that 
sharpen the triangle inequality for sums of $N$ functions in $L^p$}

\author{Eric A. Carlen}
\address[Eric A. Carlen]{Department of Mathematics, Hill Center,
Rutgers University, 110 Frelinghuysen Road, Piscataway NJ 08854-8019, USA}
\email{carlen@math.rutgers.edu}

\author{Rupert L. Frank}
\address[Rupert L. Frank] {Mathematisches Institut, Ludwig-Maximilans 
Universit\"at M\"unchen, Theresienstr. 39, 80333 M\"unchen, Germany, and 
Mathematics 253-37, Caltech, Pasa\-de\-na, CA 91125, USA}
\email{rlfrank@caltech.edu}

\author{Elliott H. Lieb}
\address[Elliott H. Lieb]{Departments of Mathematics and Physics,
Princeton University, Princeton, NJ 08544,
USA}
\email{lieb@princeton.edu}

\thanks{\copyright\, 2019 by the authors. This paper may be reproduced, in
its entirety, for non-commercial purposes.\\
Work partially supported by NSF grants DMS--1501007 (E.A.C.) and  DMS--1363432 (R.L.F.)}

\begin{abstract}
We study $L^p$ inequalities that sharpen the triangle inequality for sums of $N$ functions in $L^p$. 
\end{abstract}

\maketitle
\centerline{\version}

\section{Introduction and main theorem}

Since $|z|^p$ is a strictly convex function of $z\in \C$ for $p>1$,  for any $N\in \N$, 
$$\left|\textstyle{  \frac1N \sum_{j=1}^N z_j}\right|^p \leq   \frac1N \textstyle{ \sum_{j=1}^N |z_j|^p}\ ,$$
with equality if and only if $z_i = z_j$ for all $i,j$.
It follows immediately that for any set  $\{f_1,\dots,f_N\}$ of  measurable functions on any measure space,
\begin{equation}\label{basconv}
 \textstyle{\int  \left| \frac1N \sum_{j=1}^N f_j\right|^p \leq \frac1N  \sum_{j=1}^N\int |f_j|^p}\ ,
\end{equation}
and there is equality if and only if for almost every $x$, $f_i(x) = f_j(x)$ for all $i,j$. 

The inequality \eqref{basconv} implies the triangle inequality for the $L^p$ norms, $\|\cdot\|_p$: 
Suppose that $g$ and $h$ are two functions in $L^p$ and suppose that for some $m,n\in \N$, $\|g\|_p = m$ and $\|h\|_p = n$. Let $N = m+n$, and define $f_j = \frac1m g$ for $1 \leq j \leq m$, and $f_j = \frac1n h$ for $m+1 \leq j \leq N$, and note that each $f_j$ is a unit vector.    Then \eqref{basconv} says that
\begin{equation}\label{tri1}
\left(\frac{\|g+h\|_p}{\|g\|_p + \|h\|_p}\right)^p = \int  \left|\textstyle{  \frac1N \sum_{j=1}^N f_j}\right|^p \leq  \frac1N \sum_{j=1}^N \|f_j\|_p^p =1\ .
\end{equation}
That is, $\|g+h\|_p \leq \|g\|_p+\|h\|_p$, which is the triangle inequality in $L^p$. By homogeneity, the condition on the norms of $g$ and $h$ reduces to the ratio of these norms being rational, and then by continuity, the condition on the norms may be dropped altogether. In this elementary argument, we loose information on the cases of equality 
when $\|g\|/\|h\|$ is not rational.  If however we can sharpen \eqref{basconv}, then we can also sharpen the triangle inequality, as we show below.

In this paper we shall prove several theorems that sharpen \eqref{basconv}. First rewrite
 \eqref{basconv} as
 \begin{equation}\label{basconv2}
\| \textstyle{\sum_{j=1}^N f_j}\|_p^p \leq N^{p-1} \sum_{j=1}^N\|f_j\|_p^p\ .
\end{equation}
One case in which \eqref{basconv2} leaves much room for improvement is that in which the functions
$f_1,\dots,f_N$  satisfy  $f_if_j = 0$ for all $i\neq j$; that is, the functions  have disjoint supports. 
Then  $|f_1+\cdots + f_N|^p = |f_1|^p +\cdots |f_N|^p$, and hence in this case the factor of $N^{p-1}$ in \eqref{basconv2}
is superfluous, and 
 \begin{equation}\label{basconv3}
\| \textstyle{\sum_{j=1}^N f_j}\|_p^p \leq \sum_{j=1}^N\|f_j\|_p^p\ .
\end{equation}

{\em We seek inequalities that interpolate between \eqref{basconv2} and  \eqref{basconv3} in the sense that they sharpen 
\eqref{basconv2} and reduce to \eqref{basconv3} as $\|f_if_j\|_{p/2}^{p/2}$ goes to zero for all $i\neq j$.} Towards this end we define
\begin{equation}\label{gamde2}
\boxed{
\Gamma_p(f_1,\dots,f_N) := \frac{ \textstyle{ \left\Vert {{N}\choose{2}}^{-1}\sum_{i< j} |f_if_j|\right\Vert_{p/2}^{p/2}}}{\frac1N \sum_{j=1}^N \|f_j\|_p^p}\ .}
\end{equation}
Here and in the following, we write 
${\displaystyle \|f\|_p := \left(\textstyle{\int |f|^p }\right)^{1/p} \quad \text{for {\em all}} \ p\neq 0}$
despite the fact that for $p<1$, $\|\cdot \|_p$ is not a norm, as the notation may well suggest.

Note that
\begin{equation}\label{gamde4}
0 \leq \Gamma_p(f_1,\dots,f_N) \leq 1\ ,
\end{equation} 
with equality on the left if and only if all of the functions have disjoint support, and on the right if and only if all of the functions are equal.  Thus, for any $r > 0$, an inequality of the form
\begin{equation}\label{gamde5}
\|\textstyle{\sum_{j=1}^N f_j}\|_p^p \leq  \big[1 +(N-1)
\Gamma_p^r(f_1,\dots,f_N)\big]^{p-1}{\textstyle \sum_{j=1}^N \|f_j\|_p^p}
\end{equation}
would interpolate between \eqref{basconv2} and  \eqref{basconv3}, and sharpen the former. 

At $p=2$, there is actually an {\em identity} of this form:  Since for non-negative $f_i$,
$$\Gamma_2(f_1,\dots,f_N)  =
\frac{2 \sum_{i<j}\int f_if_j}{(N-1)\sum_{i=1}^N \|f_i\|_2^2}\ ,  $$
it follows that
\begin{equation}\label{p2id}
\|\textstyle{\sum_{j=1}^N f_j}\|_2^2 =  \big[1 +(N-1)\Gamma_2(f_1,\dots,f_N)\big]{\textstyle \sum_{j=1}^N \|f_j\|_2^2}\ ,
\end{equation}
which is \eqref{gamde5} for $p=2$ and $r=1$, except that it holds as an identity and not only an inequality. The inequality
 \eqref{gamde5},  proved here for $p>2$ and with  $r= r(N,p)$ specified below, stands in the same relation to the identity
 \eqref{p2id} that Clarkson's inequality \cite{Cl} (see also \cite{BCL}):
\begin{equation}\label{clark}
\left\Vert \frac{g+h}{2}\right\Vert_p^p  + \left\Vert \frac{g-h}{2}\right\Vert_p^p  \leq  \frac{\|g\|_p^p+\|h\|_p^p}{2}\ .
\end{equation}
stands to the  {\em Parallelogram Law}, the identity between the left and right sides of \eqref{clark} at $p=2$.  For $1 < p < 2$,
we prove that the {\em reverse} of \eqref{gamde5} is valid for $r = r(N,p)$ specified below. In this case, the inequality does not sharpen the triangle inequality. Instead, it complements it by providing a \emph{lower} bound on $\|\textstyle{\sum_{j=1}^N f_j}\|_p$.
Our main theorem is the following:

\begin{theorem}[Main Theorem]  \label{main}  For any $N\geq 2$, and any set of $N$  non-negative measurable functions $f_1,\dots,f_N$ on any measure space, and all $p\in (2,\infty)$,
\begin{equation}\label{mainin}
\boxed{ \phantom{\bigg|_{A_A}^{A^A}}
\|\textstyle{\sum_{j=1}^N f_j}\|_p^p \leq  \biggl[\,1\, +\, (N-1)
\Gamma_p(f_1,\dots,f_N)^{r(N,p)}\biggr]^{p-1}{\textstyle \sum_{j=1}^N \|f_j\|_p^p} \phantom{\bigg|_{A_A}^{A^A}}}
\end{equation}
where
\begin{equation}\label{rval1}
r(N,p) =  \frac{2N}{2N+(p-2)(2N-1)}\ .
\end{equation}
For $p\in (1,2)$ the reverse inequality is valid. 
\end{theorem}

Notice that $r(N,2) =1$, so that the inequality \eqref{mainin} reduces to the identity \eqref{p2id} at $p=2$. The case $N=2$ is special. This case was considered by us in \cite{CFIL} where \eqref{mainin} was proved with $2/p$ in place of $r(2,p)$. Note that
${\displaystyle r(2,p) = \frac{4}{4 + 3(p-2)} < \frac2p}$.
and since $\Gamma_p(f_1,f_2) \leq 1$,   the larger the exponent $r$ in \eqref{gamde5} is, the stronger the inequality is. However, as we shall show here, $N\geq 3$ is significantly different from $N=2$.  The inequality  \eqref{mainin} would reduce to the inequality found in \cite{CFIL} for $N=2$ were it possible to replace $r(N,p)$ by 
\begin{equation}\label{rval2}
\widetilde{r}(N,p) =  \frac{2N}{2N+(p-2)(2N-2)}  = \frac{N}{N+ (p-2)(N-1)}\ .
\end{equation}
We shall see below why this is possible for $N=2$, but not for $N> 2$. 

For $N=2$, there is only one pair of functions to consider. When there are more pairs, there are several choices that one might make in defining the quantity $\Gamma_p(f_1,\dots,f_N)$. Another possibility that may at first seem more natural is
\begin{equation}\label{gamde}
\boxed{
\widetilde{\Gamma}_p(f_1,\dots,f_N) = 
\frac{  {{N}\choose{2}}^{-1}\sum_{i< j} \|f_if_j\|_{p/2}^{p/2}}{\frac1N \sum_{j=1}^N \|f_j\|_p^p}\ .}
\end{equation}
By Jensen's  inequality, for $p>2$ (or $p< 0$),
$$\Gamma_p(f_1,\dots,f_N)  \leq \widetilde{\Gamma}_p(f_1,\dots,f_N)\ ,$$
and this inequality reverses when $p\in (0,2)$. Hence, for $p>2$, where inequality \eqref{mainin} sharpens the triangle inequality, it implies the corresponding inequality with $\widetilde{\Gamma}_p(f_1,\dots,f_N)$ replacing 
${\Gamma}_p(f_1,\dots,f_N)$.   Note also that for all $N$,
\begin{equation}\label{rbound}
r(N,p) > \frac{1}{p-1}\ .
\end{equation}
By combining the last remarks, we have the following immediate corollary of Theorem~\ref{main}.

\begin{corollary}[Simplified Bound]\label{maincl}  For any $N\geq 2$, and any set of $N$  non-negative measurable functions $f_1,\dots,f_N$ on any measure space, and all $p\in (2,\infty)$
\begin{equation}\label{mainclin}
\boxed{ \phantom{\bigg|_{A_A}^{A^A}}
\|\textstyle{\sum_{j=1}^N f_j}\|_p^p \leq  \biggl[1 +(N-1)\widetilde{\Gamma}_p(f_1,\dots,f_N)^{1/(p-1)}\biggr]^{p-1}{\textstyle \sum_{j=1}^N \|f_j\|_p^p}   \phantom{\bigg|_{A_A}^{A^A}}}
\end{equation}
For $p\in (1,2)$ the reverse inequality is valid. 
\end{corollary}

\begin{remark}\label{sharpmaincl} We shall show below that this inequality does not hold, uniformly in $N$ if the exponent $1/(p-1)$ is replaced by any larger value, though of course, for each $N$, we may replace it by $r(N,p)$, and the inequality is still valid.
\end{remark}

We now show how Corollary~\ref{maincl} yields a sharpened form of the triangle inequality for $p>2$.     Let $g,h\in L^p$, and define $\lambda := \frac{\|g\|_p}{\|g\|_p + \|h\|_p}$. Choose integers $m_N$ so that $\lim_{N\to\infty} \frac{m_N}{N} = \lambda$.   Define
$$u :=  \|g\|_p^{-1} |g| \, \quad   v :=  \|h\|_p^{-1} |h|   \quad{\rm and}\quad f_j := \begin{cases} u & 1 \leq j \leq m_N\\ v & m_N < j \leq N\end{cases}\ .$$
Then, reasoning  as in \eqref{tri1},
\begin{equation}\label{tri2}
\left(\frac{\||g|+|h|\|_p}{\|g\|_p + \|h\|_p}\right)^p = \lim_{N\to\infty}\int  \left|\textstyle{  \frac1N \sum_{j=1}^N f_j}\right|^p \ ,
\end{equation}
and now we apply Corollary~\ref{maincl} to estimate the right side: We find
$$
\int \left|\textstyle{  \frac1N \sum_{j=1}^N f_j}\right|^p  \leq \left(\frac{1}{N} +  \frac{N-1}{N} \left( a_N  + (1-a_N)\|uv\|_{p/2}^{p/2}  \right)^{1/(1-p)}\right)^{p-1}\ ,
$$
where 
$$
a_N = {{N}\choose{2}}^{-1}\left( {{m_N}\choose{2}}+ {{N- m_N}\choose{2}}\right) .
$$
Taking $N$ to infinity, we obtain
\begin{eqnarray}\label{tri3}
\left(\frac{\||g|+|h|\|_p}{\|g\|_p + \|h\|_p}\right)^p &\leq&  \lambda^2 + (1-\lambda)^2 + 2\lambda(1-\lambda) \|uv\|_{p/2}^{p/2} \\
 &=& 1 - 2\lambda(1-\lambda) (1- \|uv\|_{p/2}^{p/2})\\
  &=& 1 - \lambda(1-\lambda) \|u^{p/2} - v^{p/2}\|_2^2\ .
\end{eqnarray}

Expressing this in terms of $g$ and $h$, and  using $|g+h|^p \leq ||g|+|h||^p$ for $p> 0$, we have proved:
\begin{theorem}[Improved Triangle Inequality]\label{riimp} For all non-zero functions $g,h\in L^p$, $p>2$, 
\begin{equation}\label{tri5}
\|g+h\|_p \leq \left( 1 - \frac{\|g\|_p\|h\|_p}{(\|g\|_p + \|h\|_p)^2} \left\Vert \frac{|g|^{p/2}}{\|g\|_p^{p/2}}   -   \frac{|h|^{p/2}}{\|h\|_p^{p/2}}  \right\Vert_2^2 \right)^{1/p} (\|g\|_p + \|h\|_p)\ .
\end{equation}
\end{theorem}

Somewhat different stability results for the triangle inequality have been proved by Aldaz; see \cite[Theorem 4.1]{A}. His bound involves the $L^2$ distance between $\|g\|_p^{-p/2}|g|^{p/2}$ and   $\|g+h\|_p^{-p/2}|g+h|^{p/2}$
as well as between   $\|h\|_p^{-p/2}|h|^{p/2}$ and   $\|g+h\|_p^{-p/2}|g+h|^{p/2}$.  His inequality is based on a stability result for H\"older's inequality for non-negative functions. A somewhat stronger stability theorem for H\"older's inequality that does not discard information about phases by passing from $g$ and $h$  to $|g|$ and $|h|$  in the first step, was obtained in \cite{CFL2}. This may be applied with dual indices $p$ and $p/(p-1)$ to
$\|g+h\|_p^p = \int  g w + \int hw$
where $w = |g+h|^{p-1}\overline{g+h}$ to prove a variant of Aldaz' bound that does not discard phase information.

\section{Proof of Theorem~\ref{main}}

The proof of Theorem~\ref{main}, for all $N$,  is actually relatively simple compared to the proof of the slightly more incisive result for $N=2$ that is proved in \cite{CFIL}.

\begin{proof} As in our previous paper \cite{CFIL}, we replace the measure ${\rm d}\mu$ with the probability measure
$${\rm d}\nu = \frac{1}{\|\textstyle{\sum_{j=1}^N f_j\|_p^p}}|\textstyle{\sum_{j=1}^N f_j|^p{\rm d}\mu}\ ,$$
having assumed without loss of generality that $\|\textstyle{\sum_{j=1}^N f_j}\|_p < \infty$. We then replace each $f_i$ 
by $(\textstyle{\sum_{j=1}^N f_j})^{-1} f_i$. Rewriting \eqref{mainin} in terms of the new measure and the new functions, we see that it suffices to prove \eqref{mainin} under the additional assumption that the reference measure is a probability measure and the functions satisfy $\textstyle{\sum_{j=1}^N f_j =1}$ almost everywhere.  

We proceed under this assumption, first considering $p>2$.  Then
${\textstyle \sum_{i< j}f_if_j   = \frac12\left( 1 - \sum_{j=1}^N f_j^2\right)}$,
Define 
$${\textstyle B :=  \sum_{j=1}^N\int  f_j^p}\ .$$
Then \eqref{gamde5}, which is \eqref{mainin} with a non-specific value of $r$,  becomes
\begin{equation}\label{nf3}
1  \leq \left( 1 + (N-1)\left(\frac{N\int \left( \frac12{{N}\choose{2}}^{-1} \left( 1 - \sum_{j=1}^N f_j^2\right) \right)^{p/2}} { B}\right)^r\right)^{p-1} 
B\ .
\end{equation}
 By Jensen's inequality \eqref{nf3} is implied, for $p>2$, by 
\begin{equation}\label{nf4}
1  \leq \left( 1 + (N-1)\left(\frac{N \left( \frac12{{N}\choose{2}}^{-1} \left( 1 - \sum_{j=1}^N\int  f_j^2\right) \right)^{p/2}} { B}\right)^r\right)^{p-1} 
B\ .
\end{equation}
 By H\"older's inequality, 
\begin{eqnarray}\label{holder}
\textstyle{ \int f^2 =  \int f^{(p-2)/(p-1)} f^{p/(p-1)}} &\leq&  \|f^{(p-2)/(p-1)} \|_{(p-1)/(p-2)} \|  f^{p/(p-1)}\|_{p-1} \nonumber\\
&=& \|f\|_1^{(p-2)/(p-1)} (\|f\|_p^p)^{1/(p-1)}\ .
\end{eqnarray}
There is equality if and only if $f$ is constant on its support.  Using this inequality, and H\"older's inequality once more, again with exponents $(p-1)/(p-2)$ and $p-1$,
\begin{equation}\label{holder2}
\textstyle{\sum_{j=1}^N\int  f_j^2} \leq \left( \sum_{j=1}^N  \|f_j\|_1  \right)^{(p-2)/(p-1)}  \left( \sum_{j=1}^N  \|f_j\|_p^p  \right)^{1/(p-1)}   = B^{1/(p-1)}\ .  
\end{equation}

Therefore, \eqref{nf4} is implied by the inequality
\begin{equation}\label{nf5}
1  \leq \left( 1 + (N-1)\left(\frac{N \left( \frac12{{N}\choose{2}}^{-1} \left( 1 - B^{1/(p-1)}\right) \right)^{p/2}} { B}\right)^r\right)^{p-1} 
B\ 
\end{equation} in the single parameter $B$, which can take values in the range $N^{1-p}$ to $1$. (To see this fact, note that since $\textstyle{\sum_{j=1}^N f_j =1}$ almost everywhere,
$\textstyle{\sum_{j=1}^N f_j^p \leq1}$ almost everywhere, and by H\"older's inequality, 
$1= \textstyle{\sum_{j=1}^N f_j} \leq \left(\textstyle{\sum_{j=1}^N f_j^p}\right)^{1/p}N^{(p-1)/p}\ .$)

The following change of variables is useful: we write
\begin{equation}\label{cha1}
B = \left(\frac1N + x\right)^{p-1}  =  N^{1-p} \left(1 + Nx    \right)^{p-1}   \qquad {\rm for}\qquad 0 \leq x \leq \frac{N-1}{N}\ .
\end{equation}
Then
$$
\frac{1}{N(N-1)} \left( 1 - B^{1/(p-1)}\right) = \frac{1}{N^2}\left( 1 - \frac{N}{N-1} x\right)\ , 
$$
and hence
$$
N \left(\frac{1}{N(N-1)} \left( 1 - B^{1/(p-1)}\right) \right)^{p/2} = N^{1-p} \left( 1 - \frac{N}{N-1} x\right)^{p/2}\ .
$$
The inequality \eqref{nf5} is  therefore equivalent to
\begin{equation}\label{nf5B}
N \leq    \left( 1 + (N-1)\left(\frac{\left( 1 - \frac{N}{N-1} x\right)^{p/2}} { (1+Nx)^{p-1}}\right)^r\right)(1+Nx)\ .
\end{equation} 
Note that in this parameterization, we eliminate the ``outside'' power of $p-1$. 

It remains to prove \eqref{nf5B} with $r = r(N,p)$ as specified in \eqref{rval1}. With this value of $r$, 
$$\frac1r = 1 + \frac{2N-1}{2N}(p-2)$$
and therefore
\begin{equation}\label{powercal1}
p-1 -\frac1r = \frac{1}{2N}(p-2)\ .
\end{equation}
Likewise, simple computations show that
\begin{equation}\label{powercal2}
\frac{p}{2} -\frac1r   =  - \frac{(p-2)(N-1)}{2N}\ .
\end{equation}
Distributing the factor of $(1+Nx)$ that is on the right in \eqref{nf5B}, we obtain the equivalent inequality,
$$
N \leq    (1+Nx)  + (N-1)\left(\frac{\left( 1 - \frac{N}{N-1} x\right)^{p/2}} { (1+Nx)^{p-1-1/r}}\right)^r\ , 
$$
or, equivalently,
$$
\left( 1 - \frac{N}{N-1}x\right)^{1/r} \leq   \frac{\left( 1 - \frac{N}{N-1} x\right)^{p/2}} { (1+Nx)^{p-1-1/r}}
$$
which further simplifies, using \eqref{powercal1} and  \eqref{powercal2},   to 
$$\left(1+Nx\right)^{\frac{p-2}{2N}}  \leq  \left( 1 - \frac{N}{N-1}x  \right) ^{-\frac{(p-2)(N-1)}{2N}}\ ,
$$and then again to
\begin{equation}\label{final}
1+Nx   \leq  \left( 1 - \frac{N}{N-1}x  \right) ^{1-N}\ .
\end{equation}
This is trivially true for $0 < x < \frac{N-1}{N}$  by convexity. 

Now consider the case $p\in (1,2)$ so that  $p-2 < 0$, However, for all $p > 2 - \frac{2N}{2N-1}$, $r = r(N,p)$ is positive. In particular, this is true for all $p \in (1,2)$.  
Since now $p/2\in (0,1)$, Jensen's inequality again says that the the reverse of  \eqref{nf3} is implied by the reverse of \eqref{nf4}.   
The exponents in the applications for H\"older's inequality in \eqref{holder} and \eqref{holder} are  $(p-1)/(p-2)<0$ and $p-1 < 1$, and hence, integrating only over the support of $f$ in \eqref{holder}, the reverse H\"older inequality yields the reverse of \eqref{holder} and \eqref{holder2}. Thus, it remains to prove the reverse of \eqref{nf5}. We proceed as above, and since $r>0$, each step leads to the reverse of its analog, until the very last one, in which we raise both sides to the $2N/(p-2)$ power. Since this is negative, the inequality reverse, and we have reduced the reverse of \eqref{nf5} to \eqref{final} as before. 
\end{proof}

\section{Sharpness of the Inequalities}

\begin{theorem} The inequality \eqref{nf5} is valid with $r = r(N,p)$ given by \eqref{rval1}, but it is false for any larger value of $r$. 
\end{theorem}

\begin{proof} As shown in the previous section, under  the change of variables \eqref{cha1}, the inequality \eqref{nf5}  becomes the inequality \eqref{nf5B}.
To leading order in $x$,
$$
\frac{\left( 1 - \frac{N}{N-1} x\right)^{p/2}} { (1+Nx)^{p-1}}  = 1  - \alpha x + {\mathcal O}(x^2) \qquad{\rm where}\qquad \alpha = \frac{N}{2N-2}((p-2)(2N-1) + 3N)\ .
$$
Therefore, 
\eqref{nf5B}  says that
$1 \leq   [N - (N-1)r\alpha  + {\mathcal O}(x^2)][N^{-1} +x]$, and this can only hold  if
$$r \geq \frac{N^2}{N-1}\frac1\alpha = \frac{2N}{2N+(p-2)(2N-1)}\ .$$

We have already seen that \eqref{nf5B} is valid with $r = r(N,p)$ given by \eqref{rval1}. The expansion shows that this is not true for any larger value of $r$. 
\end{proof}

Because \eqref{gamde5} is a weaker inequality than  \eqref{nf5}, this leaves open the question  whether  \eqref{gamde5} could be valid for larger values of $r$,  even though \eqref{nf5} is not. We have seen that this is the case for $N=2$, and that in this case, the optimal value of $r$ is $2/p$. However, the case $N=2$ is somewhat special, as we now show. 

Let the measure space be $[0,1)$ equipped with Lebesgue measure. For $j=1,\dots,N$, let $A_j = [(j-1)/N,j/N)$.  Define
$$f_j(x) = \begin{cases} a & x\in A_j\\
\frac{1-a}{N-1} & x\notin A_j \end{cases} \quad ,\quad a\in [0,1]\ .$$
For $a= 1/N$, each $f_j$ has the constant value $1/N$.
With this choice, $\sum_{j=1}^N f_j(x)$, $\sum_{j=1}^N f_j^p(x)$ and $\sum_{i < j} f_i(x)f_j(x)$ are all identically  constant on account of symmetry in the sets $A_j$, and
\begin{equation}\label{expAA}
\sum_{j=1}^N f_j =1\ ,
\end{equation}
\begin{equation}\label{expBB}
\sum_{j=1}^N f_j^p =  a^p + \left( \frac{1}{N-1} \right)^{p-1}(1-a)^p\ ,
\end{equation}
\begin{equation}\label{expCC}
{{N}\choose{2}}^{-1}\sum_{i < j} f_if_j =   \frac{1}{N^2} - \frac{1}{(N-1)^2}(a- 1/N)^2\ ,
\end{equation}
and therefore,
\begin{equation}\label{ONP}
\textstyle{\Gamma_p(f_1,\dots,f_N)  = \left(\frac{N \left(    \frac{1}{N^2} - \frac{1}{(N-1)^2}(a- \frac1N)^2
\right)^{p/2}} { a^p + \left( \frac{1}{N-1} \right)^{p-1}(1-a)^p}\right)}\ ,
\end{equation}
Then
\eqref{gamde5} would imply
\begin{equation}\label{num}
\textstyle{
1 \leq \left(1 + (N-1)  \left(\frac{N \left(    \frac{1}{N^2} - \frac{1}{(N-1)^2}(a- \frac1N)^2
\right)^{p/2}} { a^p + \left( \frac{1}{N-1} \right)^{p-1}(1-a)^p}\right)^r \right)^{p-1} \left( \textstyle{  a^p + \left( \frac{1}{N-1} \right)^{p-1}(1-a)^p}\right) := K_{N,p}(a)}\ .
\end{equation}
and there is equality at $a=1/N$, in which case all of the $f_j$ are equal,  and as well at $a=1$, in which case all of the $f_j$ have mutually disjoint support.  For $N=2$, the function $K_{2,p}(a)$ is symmetric in $a$ about $a=1/2$, and there is also equality at $a=0$.   For $N> 2$, this is not the case.   We can now deduce several restrictions on the values of $r$ for which \eqref{gamde5} could possibly be valid in general.

One such restriction comes from the fact that since there is equality in \eqref{ONP} at $a=1/N$, this value of $a$ must at least be a local minimizer of $K_{N,p}(a)$.

A Taylor expansion of \eqref{expBB} about $a=1/N$ yields
\begin{equation}\label{expBB2}
\sum_{j=1}^N f_j^p =  N^{1-p}\left( 1 + \frac{p(p-1)}{2}\frac{N^2}{N-1}t^2  +   \frac{p(p-1)(p-2)}{6}\frac{N^3(N-2)}{(N-1)^2}t^3 + {\mathcal O}(t^4)  \right)\ ,
\end{equation}
where $t= a-1/N$.   From here one easily finds that the inequality $K_{N,p}(a) \geq 1$ implies
$$0 \leq  -  r\frac{p(p-1)}{2} \frac{((p-1)N - p+2)}{N(N-1)}t^2  + \frac{p(p-1)}{2(N-1)} t^2 + {\mathcal O}(t^3)\ ,
$$
and this is true if and only if 
${\displaystyle r \leq    \frac{N}{N+ (p-2)(N-1)}}$. 
Thus, we must take $r$ no larger than $\tilde{r}(N,p)$ as defined in \eqref{rval2}.  Note that $\tilde{r}(2,p) = p/2$, and so this necessary condition on $r$ also turns out to be sufficient for $N=2$. However, it is {\em not} sufficient for $N > 2$:  
Note that the cubic term in \eqref{expBB2} vanishes only for $N=2$. For all other $N$, when $r = \tilde{r}(N,p)$, $a=1/N$ will be an inflection point of $K_{N,p}(a)$, and not a local minimum. 

\begin{remark}\label{sharpmaincl2} The same considerations apply to the family of inequalities
\begin{equation}\label{gamde55}
\|\textstyle{\sum_{j=1}^N f_j}\|_p^p \leq  \big[1 +(N-1)
\widetilde{\Gamma}_p^r(f_1,\dots,f_N)\big]^{p-1}{\textstyle \sum_{j=1}^N \|f_j\|_p^p}\ .
\end{equation}
because for our trial functions, $\widetilde{\Gamma}_p^r(f_1,\dots,f_N) = \Gamma_p^r(f_1,\dots,f_N)$. 
Thus, \eqref{gamde55} is valid for $r = r(N,p)$, but not for $r \leq \widetilde{r}(N,p)$. Since
$$\lim_{N\to\infty}r(N,p)  =\lim_{N\to\infty}\widetilde{r}(N,p) = \frac{1}{p-1}\ ,$$
it follows that the exponent $1/(p-1)$ is optimal in Corollary~\ref{maincl}, as claimed in Remark~\ref{sharpmaincl}.
\end{remark}

The simplest way to demonstrate that things do go wrong for $r = \tilde{r}(N,p)$, $N>2$ is to compute $K_{N,p}(0)$:   $K_{N,p}(0)\geq 1$  if and only if
\begin{equation}\label{nec2}
r \leq  \frac{\ln(N-1) - \ln(N-2)}{ (\frac{p}{2}-1)\ln(N) + \ln(N-1) - \frac{p}{2}\ln(N-2)   }\ .
\end{equation}

For example, with $N=3$ and $p=4$ the right side is
\begin{equation}\label{nec1}
\frac{   \ln(2)   }{ \ln(2) + \ln(3)} = 0.386852807...
\end{equation}
However, $\tilde{r}(3,4) = \frac37   =0.428571428...$, while $r(3,4) = \frac38 = 0.375$: The sufficient value $r(3,4)$ is quite close to the nececessary value specified in \eqref{nec1}. 
  Numerical experiments show that  with $r$ given by the right side of \eqref{nec2}, the inequality 
$K_{N,p}(a) \geq 1$ is likely to be valid, but of course, the inequality $K_{N,p}(a) \geq 1$ is only a  case of the inequality \eqref{mainin} for a very special choice of the functions $f_1,\dots,f_N$.   These trial functions were chosen to make
$\sum_{j=1}^N f_j^p$ and $\sum_{i<j}f_if_j$ constant. A convexity argument was used in \cite{CFIL} to reduce to this case, but the convexity on which this reduction relied fails already for $N=3$ and $p=4$.   Thus, while it is possible that the exponent $r$ in Theorem~\ref{main} could be improved slightly for $N>2$, it cannot be improved  by much.

\section{Related Results}

Until recently, most of the results related to our main theorem have concerned the case $N=2$. That there should be some strengthening of the elementary inequality \eqref{basconv} was first suggested in 2006 by Carbery,
who proposed \cite{C} several plausible refinements for $N=2$ and
$p\geq2$, of which the strongest was
\begin{equation} \label{carb}
\int \left|f+g \right|^p \leq \left(1+ \frac{\Vert fg \Vert_{p/2}} {\Vert 
f\Vert_p \Vert g\Vert_p} \right)^{p-1}\int\left(|f|^p +|g|^p\right).
\end{equation}

There is equality  when $f=g$ as in \eqref{basconv} and also when $f$ and $g$ have disjoint support. The first proof 
of  \eqref{carb} is in \cite{CFIL}, where the following stronger result is proved. 
\begin{equation} \label{carb+}
\int \left|f+g \right|^p \leq \left(1+ \Gamma_p
^{2/p}\right)^{p-1}\int\left(|f|^p +|g|^p\right),
\end{equation}
where $\Gamma_p =\Gamma_p(f,g)$ is defined in \eqref{gamde2}.

 By the arithmetic-geometric mean inequality \eqref{carb+} is stronger that \eqref{carb}, and the difference can be significant.
Not only is \eqref{carb+}  stronger, 
 it is sharp, in the sense that  $\Gamma_p^{2/p}$ cannot be replaced by $\Gamma_p^{r}$ for any $r<2/p$.
This leaves open, however, the possibility of replacing $(1+ \Gamma_p^{2/p})^{p-1}$ by some other function of $\Gamma_p$. 
After this paper was completed we received the preprint \cite{IM} in which such an inequality was proved:
\begin{equation} \label{carb++}
\int \left|f+g \right|^p \leq \left(  
\left(\frac{1+\sqrt{1-\Gamma_p^2}}{2}\right)^{1/p} +  \left(\frac{1-\sqrt{1-\Gamma_p^2}}{2}\right)^{1/p}
\right)^{p}\int\left(|f|^p +|g|^p\right).
\end{equation}

It is not at all obvious that \eqref{carb++} is stronger than \eqref{carb+}. The inequality 
\begin{equation} \label{pre}
 \left(  
\left(\frac{1+\sqrt{1-\gamma^2}}{2}\right)^{1/p} +  \left(\frac{1-\sqrt{1-\gamma^2}}{2}\right)^{1/p}
\right)^{p}  \leq     \left(1+ \gamma^{2/p}\right)^{p-1},
\end{equation}
which is valid for all $0\leq \gamma\leq 1$, is equivalent to the inequality proved in \cite[Theorem 1.3]{CFIL}, as noted in \cite{IM}.
Indeed, this inequality for all 
 $p\geq2$ is within a few percent of  being an identity, and  thus there is little {\em numerical} difference between
\eqref{carb+} and  \eqref{carb++}. The significant difference is that the function of $\Gamma_p$ on the right side of \eqref{carb++}  is shown to be the best possible in \cite{IM} in that for all $\gamma\in [0,1]$, on any ``nice'' measure space, there are functions $f$ and $g$ for which
$\Gamma_p(f,g) = \gamma$, and such that equality holds in \eqref{carb++}. However, even knowing that \eqref{carb++} is a non-improvable result for $N=2$, one does not have a new proof of
\eqref{carb+} that does not rely on  \cite[Theorem 1.3] {CFIL}.   Although Carbery proposed \eqref{carb} only for $p>2$, the inequaites \eqref{carb+} and \eqref{carb++} have been shown to hold for all $p$, with reversal of the inequality for $p < 0$ and $p\in (1,2)$. 

Carbery's paper suggests that an extension of \eqref{carb} for $N$ functions might be possible with a right hand side that involves certain matrix norms of the $N\times N$ matrix  with entries   
$\frac{\|f_if_j\|_{p/2}} {\|f_i\|_p \|f_j\|_p}$; see section 3.5 of \cite{C}. (There appears to be a typo here; the proposal as written is not homogenous.)
His proposal  partially motivated this study. The paper \cite{IM} also contains an upper  bound on the 
$L^p$ norm of a sum of $N$ functions for $N>2$, but it is for the case $p\in (1,2)$, and is of an entirely different character than the one we present here.


\bibliographystyle{amsalpha}

\end{document}